\documentclass{amsart}
\usepackage[russian,english]{babel}

\usepackage{amsfonts}
\usepackage{amssymb}
\usepackage{color}

\newtheorem{thm}{Theorem}[section]
\newtheorem{cor}[thm]{Corollary}
\newtheorem{lem}[thm]{Lemma}
\newtheorem{prob}[thm]{Problem}

\theoremstyle{definition}

\theoremstyle{remark}
\newtheorem{rem}{Remark}[section]

\numberwithin{equation}{section}

\begin{document}

\title[Approximation of solutions]{Approximation of solutions to parabolic 
Lam\'e type operators in cylinder domains and Carleman's formulas for them 
}

\author{P.Yu. Vilkov}

\address[Pavel Vilkov]
{Siberian Federal University
                                                 \\
         pr. Svobodnyi 79
                                                 \\
         660041 Krasnoyarsk
                                                 \\
         Russia}
\email{pavel\_vilkov17@mail.ru}

\author{I.A. Kurilenko}

\address[Il'ya Kurilenko]
{Siberian Federal University
                                                 \\
         pr. Svobodnyi 79
                                                 \\
         660041 Krasnoyarsk
                                                 \\
         Russia}
\email{ilyakurq@gmail.com}

\author{A.A. Shlapunov}

\address[Alexander Shlapunov]
{Siberian Federal University
                                                 \\
         pr. Svobodnyi 79
                                                 \\
         660041 Krasnoyarsk
                                                 \\
         Russia}
\email{ashlapunov@sfu-kras.ru}

\begin{abstract}
Let $s \in {\mathbb N}$, 
$T_1,T_2 \in {\mathbb R}$, $T_1<T_2$, and let $\Omega, \omega $ be bounded
domains in ${\mathbb R}^n$, $n \geq 1$ such that $\omega \subset \Omega$ and 
the complement $\Omega \setminus \omega$ have no non-empty compact components in 
 $\Omega$. We investigate the problem of approximation of solutions to parabolic 
Lam\'e type system from the Lebesgue class 
 $L^2(\omega \times (T_1,T_2))$ in a cylinder domain 
 $\omega \times (T_1,T_2) \subset {\mathbb R}^{n+1}$ by more regular solutions in 
a bigger domain $\Omega \times (T_1,T_2)$. As an application of the obtained 
approximation theorems we construct  Carleman's formulas for recovering solutions to 
these parabolic operators from the Sobolev class  $H^{2s,s}(\Omega \times (T_1,T_2))$ 
via  values the solutions on a part of the lateral surface of the cylinder and the 
corresponding them stress tensors.
\end{abstract}

\maketitle

\section*{Introduction}
\label{s.Int}

It was known from the middle of the XX-th century that the approximation 
theorems for solutions of homogeneous elliptic equations (and systems) 
are closely related to ill-posed problems for the corresponding elliptic 
operators, see, for instance,  \cite{Merg56}, \cite{MaHa74} for the Laplace equation 
or  \cite[ch. 5--8, 10]{Tark36} for general elliptic operators with the uniqueness 
condition in small. Actually, the key role in the development of the approximation 
theory has the approach by C. Runge \cite{R1885} for the uniform approximation 
of holomorphic functions on compact sets; see also \cite{Mal56} for (non-necessarily elliptic)
operators with constant coefficients and \cite[ch. 4, 5]{Tark37} for elliptic 
operators with sufficiently smooth coefficients). However, more delicate 
approximation theorems in various function spaces, where behaviour of the elements 
are controlled up to the boundary of the considered sets, appeared to be more 
important for applications, see, for instance,  the pioneer papers by A.G. Vitushkin  
\cite{V67} and V.P. Havin \cite{H68} for analytic functions or the monograph 
\cite[ch. 5--8]{Tark36} for Sobolev solutions to systems of differential equations 
with surjective/injective symbols.

Taking in account general suggestions by 
 M.M. Lavrent'ev \cite{Lv57}, 
S. Bergman  \cite{Berg70}, I.F. Krasichkov \cite{Kras68}, a way to find  
systems with the double orthogonality property was indicated in the paper by 
L.A. Aizenberg and A.M. Kytmanov \cite{AKy}. Combined with the approximation 
theorems and the integral representation method, this approach by \cite{AKy} has 
lead to the construction of Carleman's formulas for exact and approximate solutions 
to the Cauchy problem for holomorphic functions. This scheme was successfully 
adopted for the investigation of the Cauchy problem for a wide class of elliptic
equations, see  \cite[ch. 10, 12]{Tark36}, \cite{ShSMZh},  
\cite{ShTaLMS}, \cite{ShZamm}, or elliptic complexes, see \cite{FeSh2}). 

In the last decades, the area of the application of the ill-posed Cauchy type problems 
has expanded due to the theory of parabolic equations, see, 
for instance, \cite{PuSh15}, \cite{MMT17}, 
\cite{KuSh}. In the present paper we apply  the described above scheme by L.A. Aizenberg 
in order to study the ill-posed Cauchy problem for the parabolic Lam\'e type operator  
${\mathcal L}$ in ${\mathbb R}^{n+1}$,
\begin{equation} \label{eq.L}
{\mathcal L} = \partial _t 
-\mu\Delta - (\mu + \lambda) \nabla\operatorname{div} + \sum_{j=1}^n A_j \partial_j  + A_0,
\end{equation}
where $\Delta$ is the Laplace operator in ${\mathbb R}^{n}$, $\nabla$ is the gradient 
operator, $\operatorname{div} $ is the divergence operator, $\mu$ and $\lambda$ are the 
Lam\'e constants, $\mu>0$, $\lambda+\mu \geq 0$,  and 
$A_j$ are $(n\times n)$-matrices over ${\mathbb R}$; actually, 
equations with the operator ${\mathcal L}$ can be considered as 
one of the linearisation of the evolution Navier-Stokes Equations, 
see, for instance, \cite[ch. III, \S 1]{Tema79}. 

More precisely, according to \cite{PuSh15}, the ill-posed Cauchy problem for the operator 
${\mathcal L}$ in a cylinder domain in ${\mathbb R}^{n+1}$ with the data 
on its lateral surface can be reduced to the problem of the extension 
of solutions to the operator ${\mathcal L}$ from a lesser cylinder domain 
to a bigger one. The last problem could be solved with the use of the systems 
with the double orthogonality property when considered in suitable Hilbert spaces. 
However, in \cite{PuSh15} a solvability criterion for the Cauchy problem 
was obtained in the H\"older spaces  and the solutions were constructed as formal 
power series, only.  In the present paper we obtain the solvability criterion 
in the anisotropic Sobolev spaces in terms systems 
with the double orthogonality property and, using them, we construct Carleman's 
formulas for the exact and the approximate solution to the Cauchy problem
cf. \cite{KuSh} for the heat operator. Since Aizenberg's method  relies on the completeness 
of  the used doubly orthogonal system, we also prove the theorem on the approximation 
from the Lebesgue class $L^2(\omega \times (T_1,T_2))$ in the cylinder domain 
 $\omega \times (T_1,T_2) \subset {\mathbb R}^{n+1}$
by more regular solutions in a bigger domain $\Omega \times (T_1,T_2)$, 
cf. \cite{ShHeat} for the heat operator.

\section{Preliminaries}
\label{s.1}

Let ${\mathbb R}^n$ be the $n$-dimensional Euclidean space with the coordinates  
$x=(x_1, \dots , x_n)$ and let  $\Omega \subset {\mathbb R}^n$ be a bounded domain 
(open connected set). As usual, denote by  $\overline{\Omega}$ the closure of $\Omega$, and by
 $\partial\Omega$ its boundary.  We assume that  $\partial \Omega$ 
is piece-wise smooth hypersurface. We denote also by 
$\Omega_{T_1,T_2}$ a bounded open cylinder  
 $\Omega \times (T_1,T_2)$ in ${\mathbb R}^{n+1}$ with $T_1<T_2$; 
for $\Omega \times (0,T)$ we write simply  $\Omega_T$. 
 
We consider functions over  ${\mathbb R}^n$ and 
${\mathbb R}^{n+1}$.  As usual, for $s \in {\mathbb Z}_+$ we denote by $C^s(\Omega)$ 
the space of all $s$ times continuously differentiable 
functions on $\Omega$ and by $C^{s,\gamma}(\Omega)$ we denote 
the corresponding H\"older space with the power $\gamma\in (0,1)$. Next,  for  
$\gamma\in [0,1)$ and relatively open set 
 $S \subset \partial \Omega$, we denote by
	$C^{s,\gamma}(\Omega\cup S)$ the set of functions from   
$C^{s,\gamma}(\Omega)$ such that all their derivatives up to order $s$
extend continuously on $\Omega \cup S$ (the last space will be considered 
as a Banach space if  $S=\partial \Omega$). Let  
 $L^2 (\Omega)$ be the Lebesgue space over 
$\Omega$ with the standard inner product 
$(u,v)_{L^2 (\Omega)} $, 
and $H^s (\Omega)$, $s\in \mathbb N$, be the Sobolev space  
with the standard inner product 
$(u,v)_{H^s (\Omega)}$. Investigating spaces of solutions to the heat equation, we need 
the anisotropic Sobolev spaces $H^{2s,s} (\Omega_{T_1,T_2})$, $s \in  {\mathbb Z}_+$, 
with the standard inner product and the anisotropic H\"older spaces 
$C^{2s,\gamma,s,\frac{\gamma}{2}} ((\Omega \cup S)_{T_1,T_2})$ 
(the last space will be considered as a Banach space if  $S=\partial \Omega$). 
 see, for instance, \cite{LadSoUr67}, \cite[ch. 2]{Kry08}, \cite{Kry96}.

Finally, for $k \in {\mathbb Z}_+$, we denote by $H^{k,2s,s} (\Omega_{T_1,T_2})$ 
the set of all functions  $u \in H^{2s,s} (\Omega_{T_1,T_2})$ such that 
 $\partial ^\beta_x u \in H^{2s,s} (\Omega_{T_1,T_2})$ for all  $|\beta|\leq k$. 
Similarly, one introduces that space 
$C^{k, 2s,\gamma,s,\frac{\gamma}{2}} ((\Omega \cup S)_{T_1,T_2})$. 

It is convenient to denote by  ${\mathbf C} ^{s,\gamma} (\Omega\cup S)$ 
the space of vector fields  ($n$-vector functions) with components 
of the class $C ^{s,\gamma} (\Omega\cup S)$ and, similarly, 
for the spaces $\mathbf{L}^2 (\Omega)$, 
$\mathbf{H}^s (\Omega)$, $\mathbf{H}^{2s,s} (\Omega_{T_1,T_2})$, etc.

We also will use the so-called Bocher spaces 
of functions depending on $(x,t)$ from the strip  
$\mathbb{R}^n \times  [T_1,T_2]$.
Namely, for a Banach space $\mathcal B$ (for example, the space of functions 
on a subdomain of $\mathbb{R}^n$) and    $p \geq 1$, we denote by  
$L^p (I,{\mathcal B})$ the Banach space of all the measurable mappings 
  $u : [T_1,T_2] \to {\mathcal B}$
with the finite norm  
$$
   \| u \|_{L^p ([T_1,T_2],{\mathcal B})}
 := \| \|  u (\cdot,t) \|_{\mathcal B} \|_{L^p ([T_1,T_2])},
$$
see, for instance, \cite[ch. \S 1.2]{Lion69},  \cite[ch.~III, \S~1]{Tema79}. 

The space $C ([T_1,T_2],{\mathcal B})$ is introduced with the use of the 
same scheme; this is the Banach space of all the continuous mappings
$u : [T_1,T_2] \to {\mathcal B}$ with the finite norm 
$$
   \| u \|_{C ([T_1,T_2],{\mathcal B})}
 := \sup_{t \in [T_1,T_2]} \| u (\cdot,t) \|_{\mathcal B}.
$$

Obviously, the steady differential Lam\'e type operator 
$$ 
L= -\mu\Delta - (\mu + \lambda) \nabla\operatorname{div} + \sum_{m=1}^n A_m \partial_m  + A_0
$$
is strongly elliptic, if $\mu>0$ and $\lambda+\mu\geq 0$. Hence the operator  
${\mathcal L}$, given by formula \eqref{eq.L}, is strongly uniformly 
parabolic, see, for instance, \cite{eid}, \cite{sol}; besides, as the 
coefficients of the operator are constant, it is well known, it admits a fundamental 
solution  $\Phi (x,y,t,\tau)$ of the  convolution type, i.e. there is a kernel 
$\Phi (x,t)$ that can be constructed with the use of the Fourier transform and such that $\Phi
 (x,y,t,\tau)= \Phi (x-y,t-\tau) $, see \cite[Ch. 2, \S 1.2]{eid}. In the simplest case where 
$A_m =0$ for all $m=0, 1, \dots, n$, the components $\Phi_{i,j} (x,t)$  
of the kernel $\Phi (x,t)$ of the fundamental solution may be written as 
\begin{equation*}
 \Phi_{i,j}(x,t)=\varphi(x,\mu t)\delta_{i,j}+
\int_{\mu t}^{(2\mu +\lambda) t}\frac{\partial^2\varphi 
(x,\zeta)}{\partial x_i \partial x_j}d\zeta, 
\,\, 1\leq i,j\leq n, 
\end{equation*}
where $\delta_{i,j}$ is the Kronecker symbol and  
\begin{equation*}
	\varphi (x,t)=
	\left\{
	\begin{array}{lll}
	\frac{1}{(2\sqrt{\pi t})^n}e^{-\frac{|x|^2}{4t}} & \textrm{ if } & t>0,\\
	0 & \textrm{ if } & t\leq 0, 
	\end{array}
	\right.
\end{equation*}
is  the 
fundamental solution of the heat operator $\partial _t - \Delta$, 
see \cite{eid}. Formulas for the components of the fundamental solutions 
$\varphi (x,t)$ and $\Phi (x,t)$ imply immediately that they are smooth outside 
the diagonal  $\{ (x,t ) =0\}$ and real analytic with respect to 
the space variables. In particular, this means 
that the parabolic operator ${\mathcal L}$ is hypoelliptic.  
Besides, the fundamental solution allows to construct a useful integral 
Green formula for the operator ${\mathcal L}$. With this purpose, fix  
a $(n\times n)$-matrix differential first order operator $B_1$  with the principal 
symbol non-degenerate on the normal vectors  $\nu(x) =(\nu_1(x), \dots, \nu_n(x))$ 
to the surface $\partial \Omega$ 
at all the points $x \in \partial \Omega$. Denote by $\{C_0, C_1\}$ the Dirichlet 
pair associated with the Dirichlet pair $\{I_n, B_1\}$ via (first) Green formula
for the operator $L$, i.e. $(n\times n)$-matrix differential operators $C_j$ 
of order $j$,  with the principal 
symbol non-degenerate on the normal vectors
 $\nu(x)$ and such that  
$$
\int_{\partial \Omega} \Big( (C_1 v)^* u + (C_0 v)^* B_1 u\Big) ds = 
(Lu,v) _{\mathbf{L}^2 (\Omega)} - (u,L^*v) _{\mathbf{L}^2 (\Omega) }
$$
for all $u,v \in \mathbf{C}^\infty (\overline \Omega)$, where $I_n$ is the unit 
$(n\times n)$-matrix, and $L^*$ is the formal adjoint operator for $L$. 
As the operator $B_1$ one usually takes the boundary stress tensor  $\sigma$  with the 
components
	\begin{equation*}
		\sigma_{i,j} = \mu\delta_{i,j}\sum\limits_{k=1}^{n}\nu_k\frac{\partial}{\partial x_k} + 
		\mu\nu_{j}\frac{\partial}{\partial x_i} + \lambda\nu_{i}\frac{\partial}{\partial x_j}.
	\end{equation*}
Then, in the simplest case where $A=0$, we obtain $C_0=I_n$, $C_1 = \sigma$. 
Next, for vector functions  $f \in  \mathbf{L}^2(\Omega_{T_1,T_2})$, 
$v \in L^2 ([T_1,T_2]), 
\mathbf{H}^{1/2}(\partial \Omega))$, $w \in L^2 ([T_1,T_2]), 
\mathbf{H}^{3/2}(\partial \Omega))$, $h \in \mathbf{H}^{1/2}(\Omega)$, we consider
parabolic potentials:  
\begin{equation*} 
I_{\Omega, T_1} (h) (x,t)= \int\limits_{\Omega}\Phi(x-y, t)h(y,T_1) dy,  
$$
$$
G_{\Omega,T_1} (f) (x,t)=\int\limits_{T_1}^t\int\limits_\Omega \Phi(x-y,\ t-\tau)f(y, \tau)
dy d\tau, 
\end{equation*}
\begin{equation*} 
V_{\partial \Omega,T_1} (v) (x,t)=\int\limits_{T_1}^t\int\limits_{\partial \Omega} 
(C_{0,y}\Phi ^* (x-y, t-\tau))^* (x-y, t-\tau) v(y, \tau)
ds(y)d\tau, 
\end{equation*}
\begin{equation*} 
W_{\partial \Omega,T_1} (w) (x,t)= - \int\limits_{T_1}^t\int\limits_{\partial \Omega} 
(C_{1,y} \Phi^* (x-y, t-\tau))^*  w(y, \tau)ds(y) d\tau
\end{equation*}
(see, for instance, \cite[ch. 1, \S 3, ch. 5, \S 2]{frid}. 
By the definition, these are  (improper) integrals, depending 
on the parameter $(x,t)$.

\begin{lem} \label{l.Green} 
Let $\partial \Omega \in C^2$.  
For any $ T_1 < T_2$ and all  $u \in {\mathbf H}^{2,1} (\Omega_{T_1,T_2}) $ the following 
formula holds: 
\begin{equation} \label{eq.Green}
\left.
\begin{aligned}
u(x, t) \mbox{ in } \Omega_{T_1,T_2}  \\
0 \mbox{ outside } \overline{\Omega_{T_1,T_2}} 
\end{aligned}
\right\} \! = 
I_{\Omega,T_1} (u)   + G_{\Omega,T_1} ({\mathcal L}u)  + 
V _{\partial \Omega, T_1} \left( 
B_1 u \right)  +  W_{\partial \Omega, T_1} (u) .
\end{equation}
\end{lem}

\begin{proof}  See, \cite[ch. 6, \S 12]{svesh} 
(and also \cite[theorem 2.4.8]{Tark37} for more general 
operators, admitting fundamental solutions/parametreces).
\end{proof}

Now, let $S _{\mathcal L}(\Omega_{T_1,T_2})$ be the set of all the generalized 
$n$-vector functions on $\Omega_{T_1,T_2}$, satisfying (homogeneous) Lam\'e type equation
\begin{equation} \label{eq.heat}
{\mathcal L} u = 0 \mbox{ in } \Omega_{T_1,T_2} 
\end{equation}
in the sense of distributions. First of all we note that the hypoellipticity 
of the operator $\mathcal L$ means that all the solutions to equation \eqref{eq.heat} 
are infinitely differentiable on their domain, i.e.  
$$
S _{\mathcal L}(\Omega_{T_1,T_2}) \subset \mathbf{C}^\infty (\Omega_{T_1,T_2}).
$$
As it is known, this is a closed subspace in the space 
$C^\infty (\Omega_{T_1,T_2})$ with the standard Fr\'echet topology
(inducing the uniform convergence together with all the derivatives 
on compact subsets of   $\Omega_{T_1,T_2}$). 

Also, we need the space  $S _{\mathcal L}(\overline{\Omega_{T_1,T_2}})$, 
defined as the union of the spaces 
$$
\cup_{G \supset 
\overline{\Omega_{T_1,T_2}} } S _{\mathcal L}(G),
$$
where the union is with respect to all the domains $G \subset {\mathbb R}^{n+1}$,  
containing the closure of the domain  $\Omega_{T_1,T_2}$.

Let ${\mathbf H}^{k,2s,s} _{\mathcal L}(\Omega_{T_1,T_2}) =
\mathbf{H}^{k,2s,s} (\Omega_{T_1,T_2}) \cap S _{\mathcal L}(\Omega_{T_1,T_2}) $, 
$s \in  {\mathbb Z}_+$, $k \in {\mathbb Z}_+$. 
As it is known, this is a closed subspace of the Sobolev space 
$\mathbf{H}^{k,2s,s}(\Omega_{T_1,T_2})$. 
Similarly, ${\mathbf C}^{\infty} _{\mathcal L} (\overline{\Omega_{T_1,T_2}}) =
\mathbf{C}^{\infty} (\overline{\Omega_{T_1,T_2}}) \cap S _{\mathcal L}(\Omega_{T_1,T_2}) $ 
is a closed subspace, consisting of solutions to equation \eqref{eq.heat}, in the space 
$\mathbf{C}^{\infty}  (\overline{\Omega_{T_1,T_2}})$. It follows from the 
hypoellipticity of  the operator  $\mathcal L$ that the following (continuous) 
embeddings 
\begin{equation} \label{eq.emb.H}
S _{\mathcal L}(\overline{\Omega_{T_1,T_2}}) \subset 
\mathbf{C}^{\infty} _{\mathcal L} (\overline{\Omega_{T_1,T_2}}) \subset
\mathbf{H}^{k,2s,s} _{\mathcal L}(\Omega_{T_1,T_2}) 
\end{equation} 
are fulfilled for all $k,s\in {\mathbb Z}_+$. 

\section{An approximation theorem}
\label{s.2}

In this section we discuss an approximation theorem for solutions 
to the operator $\mathcal L$. Actually, it is quite similar to 
the approximation theorems for elliptic operators mentioned in the introduction.
Also, they are well known for the heat equation,
see, for instance, \cite{J} for the spaces with uniform convergence 
on compact sets or  \cite{ShHeat} for the Sobolev type spaces. 

\begin{thm}[P.Yu. Vilkov, A.A. Shlapunov] \label{t.dense.base}
Let  $\omega \subset \Omega \Subset {\mathbb R}^n$, $\partial \omega \in C^2, \partial \Omega 
\in C^1$. If the complement $\Omega \setminus  \omega$ has no compact components in 
$\Omega$  then $S_{\mathcal L}(\overline {\Omega_{T_1,T_2}})$ 
is everywhere dense in $\mathbf{L}^{2} _{\mathcal L}(\omega_{T_1,T_2})$. Conversely, 
 if $A=0$, then the absence of the compact components of $\Omega\setminus\omega$
in $\Omega$ is also  necessary for the density of 
$S_{\mathcal L}(\overline {\Omega_{T_1,T_2}})$ 
 in $\mathbf{L}^{2} _{\mathcal L}(\omega_{T_1,T_2})$. 
\end{thm}

\begin{proof} \textit{Sufficiency.} Clearly, the set 
 $S_{\mathcal L}(\overline {\Omega_{T_1,T_2}})$ is everywhere dense in
 $\mathbf{L}^{2} _{\mathcal L}(\omega_{T_1,T_2})$ if and only if the 
following relations 
\begin{equation} \label{eq.ort}
(u,w)_{\mathbf{L}^2 (\omega_{T_1,T_2})} = 0 \mbox{ for all } w\in 
S_{\mathcal L}(\overline {\Omega_{T_1,T_2}})
\end{equation}
means precisely for  the  field $u\in \mathbf{L}^{2} _{\mathcal L}(\omega_{T_1,T_2})$ 
that $u\equiv 0$ in $\omega_{T_1,T_2}$. 

Assume that the complement 
$\Omega \setminus  \omega$ has no (non-empty connected) compact components in  $\Omega$.  In order to prove the sufficiency of the statement 
we will use the fact that the operator ${\mathcal L}$ admits the bilateral 
fundamental solution of the convolution type. By the definition, 
\begin{equation} \label{eq.right}
{\mathcal L}_{x,t} \Phi(x-y,t-\tau) = I_n \delta (x-y, t-\tau),   
\end{equation}
where $\delta (x, t)$ is the Dirac functional supported at the point $(x,t)$. 
Besides, the convolution type provides the normality property for the 
fundamental solution, i.e.
\begin{equation} \label{eq.left}
{\mathcal L}^*_{y,\tau} \textcolor{red}{\Phi^*(x-y,t-\tau)}  =I_n \delta (x-y, t-\tau). 
\end{equation}
We note that the normality property is not self-evident, though it can be provided under 
some not very restrictive assumptions,  see \cite[Property 2.2]{eid}.

Let for the field $u\in \mathbf{L}^{2} _{\mathcal L}(\omega_{T_1,T_2})$ relation 
\eqref{eq.ort} is fulfilled. Consider an auxiliary vector field 
\begin{equation}
\label{eq.v}
v (y,\tau) = \int_{\omega_{T_1,T_2}} \Phi^*(x-y,t-\tau) u (x,t)   dx \, dt.  
\end{equation}
According to \eqref{eq.right} we have
${\mathcal L}_{x,t} \Phi(x-y,t-\tau) =0 $, if  $ (x,t) \ne (y,\tau)$. 
That is why 
${\mathcal L}_{x,t} \Phi(x-y,t-\tau) =0$ in $\Omega_{T_1,T_2}$ 
for each fixed pair $(y,\tau) \not \in \Omega_{T_1,T_2}$. 
Now, using the hypoellipticity of the operator $\mathcal L$, we conclude that 
$\Phi(x-y,t-\tau) \in S_{\mathcal L}(\overline {\Omega_{T_1,T_2}})$ 
with respect to variables $(x,t) \in \Omega_{T_1,T_2}$ for each fixed 
pair  $(y,\tau) \not \in 
\Omega_{T_1,T_2}$. In particular, relations  \eqref{eq.ort}  imply 
\begin{equation} \label{eq.v1}
v (y,\tau) = 0   \mbox{ in }{\mathbb R}^{n+1} \setminus \overline \Omega_{T_1,T_2}.
\end{equation}
On the other hand, \eqref{eq.left} yields  
\begin{equation} \label{eq.v0}
{\mathcal L}^*v    = \chi_{\omega_{T_1,T_2}} u \mbox{ in } {\mathbb R}^{n+1},
\end{equation}
where $\chi_{\omega_{T_1,T_2}}$ is the characteristic function 
of the domain  $\omega_{T_1,T_2}$. 
Obviously, 
$$
{\mathcal L}^*_{y,\tau} v (y,\tau)=0 \mbox{ in  } {\mathbb R}^{n+1} 
\setminus \overline \omega_{T_1,T_2},
$$
and then, by the discussed above  properties of the fundamental solution, the vector field 
 $v$ is real analytic with respect to the space variables of 
${\mathbb R}^{n+1} \setminus \overline \omega_{T_1,T_2} $. 

Since $\omega$ has a smooth boundary, each component of   
${\mathbb R}^{n} \setminus \overline \omega $ is itself a smooth domain and 
the similar conclusion is valid for the domain $\Omega $, too. However the 
complement $\Omega \setminus  \omega$ has no compact components in 
$\Omega$ and hence each component of  
${\mathbb R}^{n} \setminus \omega $ intersects with
 ${\mathbb R}^{n} \setminus  \Omega $ by a non-empty open set. 
Thus, \eqref{eq.v1} and the uniqueness theorem for the real analytic functions 
imply that 
\begin{equation} \label{eq.v2}
v (y,\tau) = 0   \mbox{ in }{\mathbb R}^{n+1} \setminus \overline \omega_{T_1,T_2}.
\end{equation}
In addition,  \eqref{eq.v0}, \eqref{eq.v2} mean that the field 
 $v$ is a solution to the Cauchy problem 
$$
\left\{
\begin{array}{lll}
{\mathcal L}^*  v = \chi_{\omega_{T_1,T_2}} u \mbox{ in } {\mathbb R}^{n}  \times 
(-T_2-1,1-T_1),\\
 v(y,1-T_1) = 0 \mbox{ on } {\mathbb R}^{n} .\\
\end{array}
\right.
$$
Taking in account the natural relation between parabolic and backward-parabolic 
operators and using arguments from 
 \cite[ch. 2, \S 5 theorem 3]{Kry08}, we may conclude that  
$v\in \mathbf{H}^{2,1} ({\mathbb R}^{n}  
\times (-T_2-1,1-T_1))$ and the solution is unique in this class. 
The regularity of this unique solution to the Cauchy problem can be 
expressed in term of the Bochner classes, too. Namely, 
 $v\in C ( [-T_2-1,1-T_1], \mathbf{H}^{1}({\mathbb R}^{n}) )
\cap L^2 ([-T_2-1,1-T_1], \mathbf{H}^{2} ({\mathbb R}^{n}) )
$, see, for instance, \cite[ch. 3, \S 1]{Tema79}, where similar linear problems 
for Stokes' equations are considered. In particular, the vector field  
 $v$ belongs to the space 
\begin{equation} \label{eq.space.1}
 C ( [T_1-1,T_2+1], \mathbf{H}^{1}({\mathbb R}^{n}) )
\cap L^2 ([T_1-1, T_2+1], \mathbf{H}^{2} ({\mathbb R}^{n}) )
\cap \textbf{H}^{2,1} ({\mathbb R}^{n}  \times (T_1-1,T_2+1))  .
\end{equation}

\begin{lem} \label{l.dense.c0} Any vector field of type \eqref{eq.v}, satisfying 
\eqref{eq.v2}, can be approximated by the fields from $\mathbf{C}^\infty_0  (\omega_{T_2,T_2})$
in the topology of the Hilbert space $\mathbf{H}^{2,1} (\omega_{T_1,T_2})$. 
\end{lem}

\begin{proof} Let  $\partial_ \nu=\sum\limits_{j=1}^n\nu_j (x)
\partial_{x_j}$ be the derivative 
with respect to the exterior unit normal 
$\nu (x)=
(\nu_1 (x), ..., \nu_n (x))$ to the surface $\partial\Omega$ at the point$x$. 
If  $\partial\omega$ is a $C^2$-smooth surface then, as the field 
$v$ belongs to the space \eqref{eq.space.1}, we see that there are traces
$$
v_{|\partial (\omega_{T_1,T_2})} \in
\mathbf{H}^{1/2} (\partial (\omega_{T_1,T_2})), 
$$
$$
 v_{|(\partial \omega)_{T_1,T_2}} \in L^2 ([T_1, T_2], \mathbf{H}^{3/2} (\partial \omega))
, \, \, 
\partial_\nu  v_{|(\partial \omega)_{T_1,T_2}} 
\in  L^2 ([T_1, T_2], \mathbf{H}^{1/2} (\partial \omega)),
$$
cf. \cite[ch. 3, \S 7, property 7]{MikhX}. 
Moreover, by  \eqref{eq.v2}, we have 
$$
v = 0  \mbox{ on } \partial (\omega_{T_1,T_2}), 
\, \, 
\partial_ \nu v = 0  \mbox{ on } (\partial \omega)_{T_1,T_2}.
$$
Hence $v$ belongs to both the space  \eqref{eq.space.1} and the space
\begin{equation} \label{eq.space.2}
 C ( [T_1,T_2], \mathbf{H}^{1}_0(\omega) )
\cap L^2 ([T_1, T_2], \mathbf{H}^{2}_0 (\omega) ) \cap 
\mathbf{H}^{1}_0 (\omega_{T_1,T_2}) 
\cap \mathbf{H}^{2,1} (\omega_{T_1,T_2})  ,
\end{equation}
where $ \mathbf{H}^{s}_0(\omega) $ is the closure in  $ \mathbf{H}^{s}(\omega) $ of 
the space $\mathbf{C}^\infty_0 (\omega)$ consisting of infinitely smooth fields 
with compact supports in $\omega$. Now the statement of the lemma 
easily follows with the use of the standard regularisation, see, for instance,  
\cite[ch. 3, \S 5, \S 7,]{MikhX} for the isotropic Sobolev spaces. 
\end{proof}

Next, using lemma \ref{l.dense.c0} and fixing a sequence 
$\{ v_k \} \subset \mathbf{C}^\infty _0 
(\omega _{T_1, T_2})$ converging to the field 
$v$ in $\mathbf{H}^{2,1} (\omega _{T_1, T_2})$, we see that  
$$
\|u\|^2_{\mathbf{L}^2 (\omega _{T_1, T_2})} = 
(u , {\mathcal L}^* v )_{\mathbf{L}^2 (\omega _{T_1, T_2})} = 
\lim_{k \to + \infty }(u, {\mathcal L}^* v_ k )_{\mathbf{L}^2 (\omega _{T_1, T_2})}   
= 0,
$$
because ${\mathcal L} u =0 $ in $\omega _{T_1, T_2}$ in the sense of distributions. 
Thus,  $u \equiv 0$ in $\omega _{T_1, T_2}$, that was to be proved.

\textit{Necessity.} We use the arguments similar to 
that in \cite{J} for the case of the uniform approximation of the solutions 
to the heat equation, cf. also \cite{ShHeat} for the approximation in the Lebesgue 
space. Let the complement 
$\Omega \setminus  \omega$ has at least one compact component. 
As we noted before, since the domains  $\Omega, \omega$ has smooth boundaries, 
this component is the closure of a non-empty domain  $\omega^{(0)}$.
Moreover, the set $\omega \cup \overline{\omega^{(0)}}$ is a domain with smooth boundary 
in ${\mathbb R}^n$.

Fix a point  $(x_0,t_0)\in \omega^{(0)} \times (T_1,T_2)$. 

\begin{lem}
Let $A=0$. Then for each $\delta>0$ there is such a vector function 
 $v_0 \in S _{\mathcal L}({\mathbb R}^{n+1})$ that 
 $v_0(x_0, t_0) \ne 0$ and $v_0(x,t) =0$ for all $t$, $|t-t_0|\geq \delta$. 
\end{lem}

\begin{proof} For the heat operator $(\partial_ t - a \Delta)$ such a function 
$\hat{v}_0$ was constructed by A.N. Tikhonov \cite{Ti35}. Take a function
$\hat{v}_0$, depending on $t$ and $x_1$, only, and related to $a=2\mu+\lambda$, then 
the vector field $v_0$, with the first component $\hat{v}_0$ and all other components 
being zero, fits our requirements. 
\end{proof}

Next, there is an infinitely times differentiable function 
 $\phi$ supported in $\omega $ such that $\phi (x_0) \equiv 1$ in 
a neighbourhood  $U $ of the compact 
 $\overline \omega^{(0)} $.   Then the vector field $v_1 (x,t)= \phi(x) v_0(x,-t)$
is infinitely smooth in ${\mathbb R}^{n+1}$, supported in 
$\omega \times [t_0-\delta, t_0+\delta]$ and, moreover, 
$$
  {\mathcal L}^* v_1 (x,t)= (M_1 \phi (x) (M_2 v_0 (x,-t))  +  (L^* \phi (x)) v_0 (x,-t)
	\mbox{ im }   {\mathbb R}^{n+1}
$$
with some homogeneous first order differential operators $M_1$ and $M_2$, acting 
with respect to the space variables, only. 
In particular, as $ \nabla \phi =0$ in $U$, then 
\begin{equation}\label{eq.last}
  {\mathcal L}^* v_1 =0 \mbox{ in } U \times (T_1,T_2).
\end{equation}
Denote by $\Pi_0$ the orthogonal projection from $\mathbf{L}^2 (\omega_{T_1,T_2})$ 
onto  $\mathbf{L}^2_{\mathcal L} (\omega_{T_1,T_2})$. 

Using properties of the projection
$\Pi_0$, the function $v_1$ and the fundamental solution $\Phi$, 
we obtain for all $ (y,\tau) \not \in \overline{ \omega\times (T_1,T_2)}$:
$$
\int_{T_1}^{T_2} \int_{\omega} \Phi^* (x-y, t-\tau) 
(\Pi_0 {\mathcal L}^* v_1)  (x,y) dx \, dt =
$$
\begin{equation} \label{eq.non.ort}
\int_{T_1}^{T_2} \int_{\omega\cup \omega^{(0)}} 
\Phi^* (x-y, t-\tau) ({\mathcal L}^* v_1)  (x,y)  dx \, dt =v_1 (y,\tau).
\end{equation}
Therefore the function $\Pi_0 {\mathcal L}^* v_1 \in 
\mathbf{L}^2 _{\mathcal L}(\omega_{T_1,T_2})$ is not 
$\mathbf{L}^2 (\omega_{T_1,T_2})$-orthogonal to the columns 
$\Phi_j (x-x_0, t- t_0) \in L^2 _{\mathcal L}(\omega_{T_1,T_2})$
of the matrix $\Phi (x-x_0, t- t_0)$, but it is 
$L^2 (\omega_{T_1,T_2})$-orthogonal to the vector fields 
$\Phi_j (x- y, t- \tau) \in L^2 _{\mathcal L}(\omega_{T_1,T_2})$ 
for any vectors $(y,\tau) \not \in \overline {(\omega \cup \overline{\omega^{(0)}}) 
\times (T_1,T_2)}$.

If the function $u $ belongs to $S _{\mathcal L}(\overline{ \Omega_{T_1,T_2}}) $ then 
it belongs to $H^{2,1} _{\mathcal L}(\Omega'_{T'_1,T'_2}) $ 
for some numbers $T_1'<T_1<T_2<T_2'$ and a bounded domain 
$\Omega'\Supset \Omega$. Now the Green formula yields
\begin{equation*} 
\left.
\begin{aligned}
u(x, t) \mbox{ in } \Omega'_{T'_1,T'_2}  
\\
0 \mbox{ outside } \overline{\Omega'_{T'_1,T'_2}} 
\end{aligned}
\right\} \! = 
I_{\Omega',T'_1} (u)    + 
V _{\partial \Omega, T'_1} \left( 
B_1 u \right)  +  W_{\partial \Omega', T'_1} (u). 
\end{equation*} 
According to \eqref{eq.last} ${\mathcal L}^* v_1 =0$  in 
$ \omega^{(0)} \times (T_1,T_2)$, and then  
Fubini theorem and formulas \eqref{eq.non.ort} for 
$(y,\tau) \in (\partial \Omega' \times [T_1',T'_2]) \textcolor{red}{\cup} 
(\Omega' \times \{\textcolor{red}{T_1'}\})$ imply that
$$
(\Pi_0 {\mathcal L}^* v_1,u)_{\mathbf{L}^2 _{\mathcal L}(\omega_{T_1,T_2})} = 
({\mathcal L}^* v_1,I_{\Omega',T'_1} (u)    + 
V _{\partial \Omega, T'_1} \left( 
B_1 u \right)  +  W_{\partial \Omega', T'_1} (u))_{\mathbf{L}^2 
_{\mathcal L}(\omega_{T_1,T_2})} = 
$$
$$
\int_{\Omega'} u^* (y)\Big(
\int_{T_1}^{T_2} \int_{\omega\cup \omega^{(0)}} 
\Phi^* (x-y, t-T_1') ({\mathcal L}^* v_1)  (x,y)  dx \, dt \Big) dy + 
$$
$$
\int\limits_{T'_1}^t\int\limits_{\partial \Omega'} (B_1u)^* (y) C_{0,y}\Big(
\int_{T_1}^{T_2} \int_{\omega\cup \omega^{(0)}} 
\Phi^* (x-y, t-\tau) ({\mathcal L}^* v_1)  (x,y)  dx \, dt \Big) dy +
$$
$$
\int\limits_{T'_1}^t\int\limits_{\partial \Omega'} u^* (y) C_{1,y}\Big(
\int_{T_1}^{T_2} \int_{\omega\cup \omega^{(0)}} 
\Phi^* (x-y, t-\tau) ({\mathcal L}^* v_1)  (x,y)  dx \, dt \Big) dy =0.
$$
Thus, the non-zero vector function  $\Pi_0 {\mathcal L}^* v_1 \in 
\mathbf{L}^2 _{\mathcal L}(\omega_{T_1,T_2})$ is 
$\mathbf{L}^2 (\omega_{T_1,T_2})$-orthogonal to any vector field from  
$S _{\mathcal L}(\overline{ \Omega_{T_1,T_2}})$. This proves that 
$S _{\mathcal L}(\overline {\Omega_{T_1,T_2}})$ is not everywhere dense set in the space 
$\mathbf{L}^2 _{\mathcal L}(\omega_{T_1,T_2})$ if there is a compact
component of the set $\Omega \setminus \omega$ in $\Omega$.
\end{proof}

\begin{cor}[P.Yu. Vilkov, A.A. Shlpaunov] \label{c.dense.base0}
Let $s,k\in {\mathbb Z}_+$,  
$\omega \subset \Omega \Subset {\mathbb R}^n$, 
$\partial \omega \in C^2, \partial \Omega 
\in C^1$ and let the complement have no compact components
$\Omega \setminus \omega$ in $\Omega$. Then the spaces  
$C^\infty_{\mathcal L}(\overline {\Omega_{T_1,T_2}})$ and 
$H^{k,2s,s} _{\mathcal L}(\Omega_{T_1,T_2})$ 
are everywhere dense in  $L^{2} _{\mathcal L}(\omega_{T_1,T_2})$. 
\end{cor}

\begin{proof} Follows immediately from theorem \ref{t.dense.base}, 
because of embeddings \eqref{eq.emb.H}.
\end{proof}

\section{Carleman's formulas}

In this section we consider the problem of recovering solutions to the parabolic 
Lam\'e type system in a cylinder domain via their  Cauchy data on the lateral side 
of the cylinder, see, for instance, \cite{PuSh15}. More precisely, 
let $\Gamma$ be a relatively open connected subset 
of the surface $\partial \Omega$ with a smooth boundary $\partial \Gamma$. 
  
\begin{prob} \label{pr.Cauchy}
Given vector fields $u_{1}
\in \mathbf{C}^{1,0}(\overline \Gamma \times  [0,T])$, 
$u_{2} \in \mathbf{C}^{0,0}(\overline \Gamma \times [0,T])$, $f 
\in \mathbf{C}^{0,0}(\overline \Omega 
\times [0,T]) $, find a vector field		$u\in \mathbf{C}^{2,1}(\Omega \times  [0,T]) \cap 
\mathbf{C}^{1,0}(\Gamma \times [0,T])$, 
	satisfying 
\begin{equation} \label{eq.Cauchy}
\left\{	
\begin{array}{lll}
{\mathcal L}u = f & \rm{in} &  \Omega  \times  (0,T) \\
u(x,t) = u_{1}(x,t) &  \rm{on} &  \Gamma  \times  (0,T), \\
B_1u(x,t) = u_{2}(x,t) &  \rm{on} &  \Gamma  \times  (0,T) .\\
\end{array}
\right.
\end{equation}
\end{prob}

It was proved in \cite{PuSh15} that Problem  \ref{pr.Cauchy} 
has no more than one solution (if relative interior of $\Gamma$ 
is not empty), it is densely solvable  (if the relative interior of 
$\partial \Omega \setminus \Gamma$ on  $\partial \Omega $ is not empty) 
and it is ill-posed in the sense of Hadamard and 
Also a solvability criterion in H\"older space  was obtained 
in \cite{PuSh15} for  Problem  \ref{pr.Cauchy} but, unfortunately, 
there were no Hilbert space methods involved for investigation 
of solvability conditions and the  construction of Carleman's formulas for 
the problem. We showed how to do the last two steps for a similar 
problem related to the heat equation in paper \cite{KuSh}; in the this section we apply 
these ideas to the Lam\'e type operator. For this purpose we 
need to increase essentially the smoothness of the surface  $\partial \Omega$ 
and the boundary data $u_1$, $u_2$. More precisely, we need the following lemma.  

\begin{lem} \label{l.ext}
Let $\gamma\in (0,1)$, $\partial \Omega \in C^{3+\gamma}$ and let $\Gamma$ be 
relatively open non-empty connected set on 
$\partial \Omega$ with boundary 
 $\partial \Gamma \in C^{2+\gamma}$. If 
 $u_1 \in C^{2,1,\gamma,\gamma/2} 
(\overline {\Gamma_T})$, $u_2 \in C^{2,1,\gamma,\gamma/2} 
(\overline {\Gamma_T})$ then there exist vector fields 
$\tilde u_j\in  \mathbf{C}^{2,1,\gamma,\gamma/2} (\partial \Omega_T)$ 
and $\tilde u \in \mathbf{C}^{2,1,\gamma,\gamma/2} (\overline{\Omega_T)}$
such that  $\tilde u_j =u_j $ on $\overline {\Gamma_T}$, $j=1,2$, and   
 $\tilde u =\tilde u_1 $ on $(\partial \Omega)_T$, 
$B_1 u =\tilde u_2 $ on $(\partial \Omega)_T$.
\end{lem}

\begin{proof} It is quite similar to \cite[lemma 4]{KuSh}; the only difference 
is that instead of the Dirichlet pair $(1,\partial_\nu)$ on  $\partial \Omega$ 
and the the bi-Laplacian $\Delta^2$ one has to consider the Dirichlet pair $(1,B_1)$, 
and the strongly elliptic operator $L^2$. 
\end{proof}
 
Under the assumptions of lemma  \ref{l.ext}, we set 
\begin{equation} \label{eq.tildeF}
\tilde {\mathcal F}=G_{\Omega, 0} (f)+V_{\partial \Omega,0} (\tilde u_2) + 
W_{\partial \Omega,0} (\tilde u_1) + I_{\Omega,0} (\tilde u).
\end{equation}

Assuming that $\Gamma$ is a relatively open connected set on 
 $\partial \Omega$, we find an open set 
 $\Omega^+ \subset {\mathbb R}^n$ such that the set 
$D=\Omega\cup\Gamma\cup\Omega^+$ is a domain with piece-wise smooth boundary;
it is convenient to denote  $\Omega^- = \Omega$. Then for a vector function 
$v$ on $D_T$ we denote by  $v^+$ its restriction to $\Omega^+_T$ and, similarly,
we denote by   $v^-$  its restriction to $\Omega_T$. It is natural to denote the 
boundary values (traces)  $v^\pm$ on $\Gamma_T$, if defined, by  $v^\pm_{|\Gamma_T}$. 

\begin{thm}[Kurilenko I.A.] \label{c.sol.n}
Let $\partial \Omega \in C^{3+\gamma}$, and let $\Gamma$ be a relatively open 
connected set on  $\partial \Omega$ with the boundary 
 $\partial \Gamma $ of the class $ C^{2+\gamma}$, such that 
$\partial \Omega\setminus \Gamma $ have non-empty interior on $\partial \Omega$. If 
$f\in \mathbf{C}^{0,0,\gamma,\gamma/2}
(\overline{\Omega_T})$,   $u_1\in \mathbf{C}^{2,1,\gamma,\gamma/2}(\overline{\Gamma_T})$, 
$u_2\in \mathbf{C}^{2,1,\gamma,\gamma/2}(\overline{\Gamma_T})$  then  
problem \ref{pr.Cauchy} is solvable in the space 
 $\mathbf{C}^{2,1,\gamma,\gamma/2}(\Omega_T)\cap \mathbf{C}^{1,0,
\gamma,\gamma/2}(\Omega_T \cup \Gamma_T) \cap \mathbf{H}^{2,1}(\Omega_T)$ 
if and only if there exists a vector field 
$\tilde F\in \mathbf{C}^\infty (D_T) \cap \mathbf{H}^{2,1}(D_T) \cap S_{\mathcal L} (D_T)$, 
satisfying $\tilde F=\tilde{\mathcal F}$ in $\Omega^+_T$.
\end{thm}

\begin{proof} We slightly modify the arguments from \cite[theorem 5]{PuSh15}, where 
a solvability criterion for problem \ref{pr.Cauchy} was obtained in terms of the potential 
\begin{equation} \label{eq.FF1}
{\mathcal  F} (x,t) = G_{\Omega, 0} (f)+V_{\overline{\Gamma},0} (u_2) +
W_{\overline{\Gamma},0} (u_1) .
\end{equation} 
Namely, it was proved in  \cite[theorem 5]{PuSh15} that problem 
is solvable \ref{pr.Cauchy} if and only if there is a vector field 
$ F\in \mathbf{C}^{2,1}(D_T) \cap S_{\mathcal L} (D_T)$ satisfying
$F^+= {\mathcal F}^+$  in  $\Omega^+_T$ and, moreover, if  the (unique) solution  $u$ exists 
then it is given by the following formula 
\begin{equation} \label{eq.sol}
u(x,t) = {\mathcal F}(x,t) - F(x,t)\,\, (x,t) \in \Omega_T,
\end{equation} 
and the (unique) extension $F$ of the potential ${\mathcal F}$, if exists, 
can be expressed via $u$ as 
\begin{equation} \label{eq.FF2}
F(x,t)  = {\mathcal F}(x,t) - \chi_{\Omega_T} u(x,t) \,\, (x,t) \in D_T,
\end{equation} 
where $\chi_{\Omega_T}$ is the characteristic function of the domain $\Omega_T$. 
Unfortunately, the potential  ${\mathcal F}$ is not regular enough ne the surface 
$\partial \Gamma$. However, the potential $\tilde {\mathcal F}$, defined 
with the use of lemma \ref{l.ext}, helps to improve the situation.

More precisely, by Green formula  \eqref{eq.Green}, we obtain 
  $\tilde{\mathcal F} = G_{\Omega, 0} (f-{\mathcal L}\tilde u) + \chi_{\Omega_T} \tilde u$.
Since under the assumptions of the theorem,  
$f, {\mathcal L}\tilde u \in C^{0,0,\gamma,\gamma/2}
(\overline{\Omega_T})$,then the results \cite[ch. 4, \S\S 11-14]{LadSoUr67}, 
\cite[ch. 1, \S 3]{frid} imply that 
\begin{equation} \label{eq.G1} 
G_{\Omega, 0} (f-{\mathcal L} \tilde u) \in C^{2,1,\gamma,\gamma/2}(\overline{\Omega^\pm_T}) 
\cap  C^{1,0,\gamma,\gamma/2}(D_T)  ,
\end{equation} 
i.e. $\tilde{\mathcal F} \in C^{2,1,\gamma,\gamma/2} (\overline{\Omega^\pm_T})$. 
On the other hand, 
\begin{equation} \label{eq.difference}
\tilde{\mathcal F} -{\mathcal F} = 
V_{\partial \Omega\setminus \Gamma,0} (\tilde u_2) + 
W_{\partial \Omega \setminus \Gamma,0} (\tilde u_1) + I_{\Omega,0} (\tilde u).
\end{equation}
This means that the vector field  $\tilde{\mathcal F} -{\mathcal F}$ satisfies
equation ${\mathcal L} (\tilde{\mathcal F} -{\mathcal F}) =0$ in $D_T$, and therefore
the potential ${\mathcal F}$ extends from $\Omega^+_T$ onto $D_T$ as a solution 
to the operator ${\mathcal L}$ if and only if the potential $\tilde {\mathcal F}$ extends 
from $\Omega^+_T$ onto $D_T$ as a solution to the operator ${\mathcal L}$. 

If problem \ref{pr.Cauchy} is solvable in the space 
$\mathbf{C}^{2,1,\gamma,\gamma/2}(\Omega_T)\cap \mathbf{C}^{1,0,
\gamma,\gamma/2}(\Omega_T \cup \Gamma_T) \cap \mathbf{H}^{2,1}(\Omega_T)$ then
formulas  \eqref{eq.FF1}, \eqref{eq.FF2} and \eqref{eq.difference} imply  
\begin{equation*}
\tilde F = \tilde{\mathcal F} -\chi_{\Omega_T} u \in \mathbf{H}^{2,1} (\Omega^\pm_T)
\mbox{ and } {\mathcal L} \tilde F =0 \mbox{ in } D_T.
\end{equation*}
As $\tilde F \in \mathbf{H}^{2,1} (\Omega^\pm_T)\cap C^\infty (D_T)$ (this follows, for 
instance, from \cite[ch. VI, \S 1, theorem 1]{MikhX}), we conclude that  
$\tilde F \in \mathbf{H}^{2,1} (D_T) \cap S_{\mathcal L} (D_T)$.

If there is a vector field  $\tilde F \in \mathbf{H}^{2,1} (D_T) \cap S_{\mathcal L} (D_T)$, 
coinciding  with the potential  $\tilde{\mathcal F}$ in $\Omega^+_T$, then the potential  
${\mathcal F}$ extends from $\Omega^+_T$ onto $D_T$ as a solution 
to the operator ${\mathcal L}$, i.e. problem \ref{pr.Cauchy} is solvable.
 
Moreover, it follows from formulas \eqref{eq.sol} and
\eqref{eq.difference} that in  $D_T$ we have 
\begin{equation} \label{eq.sol.tilde}
U= {\mathcal F} - F = \tilde {\mathcal F} - \tilde F \in 
\mathbf{H}^{2,1} (\Omega^\pm_T),
\end{equation}
and then $U^-$ is the unique solution to problem  \ref{pr.Cauchy} in the space 
$\mathbf{C}^{2,1,\gamma,\gamma/2}(\Omega_T)\cap \mathbf{C}^{1,0,
\gamma,\gamma/2}(\Omega_T \cup \Gamma_T) \cap \mathbf{H}^{2,1} (\Omega_T)$ 
by \cite[theorem 5]{PuSh15}. 
\end{proof}

On the other hand, theorem \ref{t.dense.base} allows to prove the existence 
of a basis with the double orthogonality property in the spaces of solutions 
to the operator $\mathcal L$ that we will use in order 
to construct formulas for the precise and approximate solutions 
to problem \ref{pr.Cauchy}.

\begin{cor}[Kurilenko I.A.] \label{c.bdo} 
Let $s\in \mathbb N$, $k \in {\mathbb Z}_+$, $\partial \Omega\in  C^1$ 
and let  $\omega$ be a relatively compact  
subdomain in $\Omega \Subset {\mathbb R}^n$ such that 
$\partial \omega \in C^2$ and the complement $\Omega\setminus \omega$ have no   
compact components in $\Omega$. Then there is an orthonormal basis 
 $\{ b_\nu\}$ is the space $\mathbf{H}^{k,2s,s}_{\mathcal L} 
(\Omega_{T_1,T_2})$ such that its restriction 
$\{ b_{\nu|\omega_{T_1,T_2}}\}$ to $\omega_{T_1,T_2}$ is on orthonormal  
basis in the space  $\mathbf{L}^{2}_{\mathcal L} (\omega_{T_1,T_2})$. 
\end{cor} 

\begin{proof} By the definition, the space 
$\mathbf{H}^{k,2s,s} _{\mathcal L} (\Omega_{T_1,T_2})$  is embedded 
continuously into the space  
 $\mathbf{L}^2 _{\mathcal L} (\omega_{T_1,T_2})$.  We denote by $R_{\Omega,\omega}$ the 
natural embedding operator  
$$
R_{\Omega,\omega}: \mathbf{H}^{k,2s,s} _{\mathcal L}(\Omega_{T_1,T_2})\to 
\mathbf{L}^2 _{\mathcal L}(\omega_{T_1,T_2}).
$$ 
The analyticity of solutions to the operator  $\mathcal L$ with respect to the space 
variables implies that the operator $R_{\Omega,\omega}$ is injective.
Besides, it follows from theorem \ref{t.dense.base} that the range of the operator 
$R_{\Omega,\omega}$ is everywhere dense in the space 
$L^2 _{\mathcal L}(\omega_{T_1,T_2})$. 

By Fubini theorem, anisotropic Sobolev space   
$\mathbf{H}^{2,1} _{\mathcal L}(\Omega_{T_1,T_2})$ is embedded continuously 
to the Bochner space 
${\mathcal B}((T_1,T_2, \mathbf{H}^2(\Omega),\mathbf{L}^2 (\Omega))$, consisting of
mappings $v: [T_1,T_2] \to \mathbf{H}^2  (\Omega)$ such that 
$\partial_t v\in \mathbf{L}^2 (\Omega)$, see \cite[ch. 1, \S 5]{Lion69}. 
By Rellich-Kondrashov theorem the embedding  
$\mathbf{H}^2 (\Omega) \to L^2 (\Omega)$ is compact. Using 
famous compact embedding theorem for the Bochner type spaces o 
(see, for example, \cite[ch. 1, \S 5, theorem 5.1]{Lion69}), we conclude that 
the space ${\mathcal B}((T_1,T_2, \mathbf{H}^2(\Omega),
\mathbf{L}^2 (\Omega))$ is embedded compactly 
into  $\mathbf{L}^2 ((T_1,T_2),\Omega) = 
\mathbf{L}^2 (\Omega_{T_1, T_2})$. Thus, the space 
$\mathbf{H}^{2,1} _{\mathcal L}(\Omega_{T_1, T_2})$  is embedded compactly into 
$\mathbf{L}^2_{\mathcal L} (\Omega_{T_1, T_2})$, and the into 
$\mathbf{L}^2 _{\mathcal L}(\omega_{T_1,T_2})$.
Therefore the space  $\mathbf{H}^{k,2s,s} _{\mathcal L}(\Omega_{T_1, T_2})$  
is embedded to $\mathbf{L}^2_{\mathcal L} (\omega_{T_1, T_2})$ compactly, too, i.e. 
the operator $R_{\Omega,\omega}$ is compact. 

Finally,   \cite[example 1.9]{ShTaLMS} implies that the complete system of eigen-vectors 
of the  compact self-adjoint operator
 $R_{\Omega,\omega}^* R_{\Omega,\omega}: 
\mathbf{H}^{k,2s,s} _{\mathcal L}(\Omega_{T_1,T_2}) \to 
\mathbf{H}^{k,2s,s} _{\mathcal L}(\Omega_{T_1,T_2})$ 
is the basis looked for; here  
$R_{\Omega,\omega}^*$ is the adjoint operator for $R_{\Omega,\omega}$ 
in the sense of the Hilbert space theory.
\end{proof}

Since in theorem \ref{c.sol.n} we have $\partial \Omega \in C^{3+\gamma}$ and 
$\partial \Gamma \in C^{2+\gamma}$, then there is a bounded domain 
$\Omega^+ \subset {\mathbb R}^n$, such that 
the domain $D=\Omega^+ \cup \Gamma \cup \Omega$ has the boundary of the class $C^{1}$. 
Fix a relatively compact subdomain $\omega$ in $\Omega^+$ 
with $\partial \omega \in C^2$  and a basis 
$\{ b_\nu\}$ with the double orthogonality property granted by corollary
 \ref{c.bdo} for the pair  $\omega$, $D=\Omega^+ \cup \Gamma \cup \Omega$. 
Consider Carleman's type kernel:
\begin{equation*}
{\mathfrak C}^{(\omega)}_{N} (x,y,t,\tau) = 
\Phi (x-y,t,\tau) -  \sum_{\nu=0}^N 
\Big( \frac{b_\nu (x,t)}{\|b_\nu\|^2 _{\mathbf{L}
^{2} _{\mathcal L}(\omega_T)}} \int_{\omega_T} 
\Phi ^* (z-y,\tilde \tau-\tau) b_\nu (z,\tilde \tau)  
 dz d \tilde \tau
\Big) .
\end{equation*}
Let also $c_\nu (\tilde{\mathcal F})$ be the Fourier coefficients 
of the potential  $\tilde{\mathcal F}$ with respect to the basis 
$\{ b_{\nu|\omega_T}\}$ in the space $\mathbf{L}^{2} _{\mathcal L}(\omega_T)$:
\begin{equation}\label{eq.Fourier}
c_\nu (\tilde{\mathcal F}) = 
\Big( \int_{\omega_T} \tilde {\mathcal F}^* (z,\tilde 
\tau)  b_\nu (z,\tilde \tau) dz d\tilde \tau 
\Big) /\|b_\nu\|^2 _{L^{2} _{\mathcal L}(\omega_T)}.
\end{equation}

\begin{cor}[Kurilenko I.A.] \label{c.sol.bdo} 
Under assumptions of theorem \ref{c.sol.n}, if $\partial D \in C^1, \partial 
\omega \in C^2$, then problem
 \ref{pr.Cauchy} is solvable in the space
$\mathbf{C}^{2,1,\gamma,\gamma/2}(\Omega_T)\cap \mathbf{C}^{1,0,
\gamma,\gamma/2}(\Omega_T \cup \Gamma_T) \cap \mathbf{H}^{2,1}(\Omega_T)$
if and only if the series  
 $\sum_{\nu=1}^\infty |c_\nu (\tilde{\mathcal  F})|^2$ converges. Besides, if exists, 
the solution to the problem is given by the following formula:
\begin{equation} \label{eq.Carleman.bdo}
u (x,t) = \lim_{N\to +\infty}
\left(\int\limits_{\Omega_t} {\mathfrak C}^{(\omega)}
_{N} (x,y,t,\tau)f(y, \tau) dy d\tau 
+ \int\limits_{(\partial \Omega)_t} 
{\mathfrak C}_{N}^{(\omega)} (x,y,t,\tau) \tilde u_2(y, \tau) 
ds(y)d\tau+ \right.
\end{equation}
\begin{equation*}
\left.
\int\limits_{(\partial \Omega)_t} (B_1 (y) {\mathfrak C}_{N} ^{(\omega)})^*
(x,y,t,\tau)  \tilde u_1(y, \tau) ds(y)d\tau \right). 
\end{equation*}
\end{cor}

\begin{proof} As we already noted, $\mathbf{H}^{2,1}_{{\mathcal L}}(\Omega_T) 
\subset C^\infty (\Omega_T)$. Moreover, as the solutions 
to the operator ${\mathcal L}$ are real analytic with respect to the space 
variables, then the solvability conditions from theorem \ref{c.sol.n} are 
equivalent to the following: 
$R_{D,\omega} \tilde F=\tilde{\mathcal F}$. Thus the first statement of the corollary 
follows immediately from \ref{c.bdo} and \cite[example 1.9]{ShTaLMS}. On the other hand,  
according to \cite[example 1.9]{ShTaLMS}, if 
$\tilde F \in H^{2,1} _{{\mathcal L}}(D_T)$ is the extension of the 
vector field  $\tilde{\mathcal F}$ from $\omega_T$ onto $D_T$ then   
\begin{equation*}
\tilde F=\sum_{\nu=0}^\infty c_\nu (\tilde{\mathcal F})b_\nu (x,t), 
\quad (x,t) \in D_T.
\end{equation*}
Hence  \eqref{eq.sol.tilde} yields 
\begin{equation} \label{eq.sol.bdo}
u (x,t) = \lim_{N\to +\infty}
\Big(\tilde {\mathcal F} (x,t) - \sum_{\nu=0}^N c_\nu (\tilde{\mathcal F})b_\nu (x,t) \Big),
\quad (x,t) \in \Omega_T.
\end{equation}
Note that if  $y \in \Omega$ and $x\in \omega$, then $x\ne y$ and the components 
of the kernel $\Phi (x-y, t-\tau)$ are integrable over $\Omega_T \times \Omega_T$. 
Hence we may use integral formula 
 \eqref{eq.tildeF} for $\tilde{\mathcal F}$ and Fubini theorem  to change 
the order the integration in \eqref{eq.Fourier}. Thus,   \eqref{eq.sol.bdo} implies 
that formula  \eqref{eq.Carleman.bdo} is true. 
\end{proof}

\begin{rem} Note that the approximation theorem  
\ref{t.dense.base}  and the results of this paper related to 
the ill-posed Cauchy problem  \ref{pr.Cauchy} may be easily 
adopted to the backward parabolic operator 
${\mathcal L}^*$, if we will use instead of 
Green formula  \ref{eq.Green} its backward parabolic analogue, 
including an integral on final data instead of the integral on the initial ones.
\end{rem}

\bigskip

The first author was supported 
by the Krasnoyarsk Mathematical Center and financed by the Ministry of Science and Higher 
Education of the Russian Federation (Agreement No. 075-02-2022-876). 
The second and the third authors were supported by a grant of the Foundation for the 
advancement of theoretical physics and mathematics ``BASIS''.

\end{document}